\documentclass[a4paper,leqno]{amsart}
\usepackage[english]{babel}

\usepackage{amsmath,amstext,amssymb,amsthm}
\usepackage{mathrsfs}
\usepackage{bbm}
\usepackage{enumitem}

\usepackage{hyperref}
\hypersetup{colorlinks = true, urlcolor = magenta, linkcolor = blue, citecolor = red}

\newcommand*\RR{\mathbb{R}}
\newcommand*\CC{\mathbb{C}}
\newcommand*\NN{\mathbb{N}}
\newcommand*\ZZ{\mathbb{Z}}

\newcommand*\CG{C_G}

\newcommand*{\tc}{\mathrel{:}}

\newcommand*\dd{\mathrm{d}}
\newcommand*\al{\alpha}
\newcommand*\be{\beta}

\newcommand*\hOC{\mathcal{H}^{\rm OC}}
\newcommand*\HDu{\mathcal{H}^{\rm D}}

\newcommand*\Eucl{E}

\theoremstyle{plain}
\newtheorem{mainthm}{Theorem}
\newtheorem{thm}{Theorem}[section]
\newtheorem{lm}[thm]{Lemma}
\newtheorem{prop}[thm]{Proposition}

\theoremstyle{remark}
\newtheorem{rem}[thm]{Remark}

\numberwithin{equation}{section}

\usepackage[dvipsnames]{xcolor}

\newcounter{comcount}

\begin{document}
\title[Sharp heat kernel estimates]{Sharp estimates for the Jacobi and trigonometric Dunkl heat kernels}

\author[P. Plewa]{Pawe\l{} Plewa}
\address[P. Plewa]{Department of Pure and Applied Mathematics, 
	Wroc{\l}aw University of Science and Technology\\ wyb. Wys\-pia{\'n}\-skie\-go  27\\ 50–-370 Wroc{\l}aw\\ Poland }        
\email{pawel.plewa@pwr.edu.pl}

\begin{abstract}
We prove sharp both-sided estimates for the $W$-invariant trigonometric Dunkl heat kernel in rank one case, including root systems $A_1$ and $BC_1$. Consequently, we obtain sharp bounds for the non-compact Jacobi heat kernel in the full range of the type parameters. 
\end{abstract}

\subjclass[2020]{58J35, 33C52} 

\keywords{Opdam--Cherednik Laplacian, trigonometric Dunkl setting, non-compact Jacobi functions, sharp heat kernel estimates, parabolic minimum principle}

\thanks{The author was partially supported by the National Science Centre (NCN) Poland, grant no. 2025/59/B/ST1/01786.}

\maketitle

\section{Introduction}

The study of harmonic analysis associated with root systems has its historical origins in the theory of Riemannian symmetric spaces of non-compact type. While classical harmonic analysis on symmetric spaces $G/K$ is inherently tied to discrete, geometrically determined root multiplicities, the groundbreaking work of Heckman and Opdam extended this framework to arbitrary root systems with continuous multiplicity parameters. Later, Cherednik introduced commuting first-order differential-difference operators, known as the trigonometric Dunkl operators. We refer the reader to the foundational works of Heckman, Opdam and Cherednik \cite{Che91, Che94,Hec90,Hec91,HecOpd,Opd96}. 

Central to this framework is the Opdam--Cherednik Laplacian $\Delta^{\rm OC}$, a second-order differential-difference operator that plays the role of the Laplace-Beltrami operator. A major goal in this area is to understand the associated heat semigroup $e^{t\Delta^{\rm OC}}$ and its integral kernel, called the Opdam--Cherednik heat kernel or the trigonometric Dunkl kernel. The heat kernel provides the fundamental solution to the heat equation, and it serves as a primary tool for studying potential theory and Riesz transforms.

Finding sharp, pointwise estimates for the Dunkl heat kernels (rational and trigonometric) is a challenging problem. In the rational Dunkl setting, Dziuba\'nski and Hejna \cite{DziHej_CVPDE'23} showed qualitatively sharp upper and lower estimates for the heat kernel. Qualitatively sharp means, that the constants appearing in the exponents of lower and upper bounds are different, therefore the polynomial terms can be ignored. The genuinely sharp estimates have been known for some time in the one-dimensional case and for the symmetric spaces of non-compact type \cite{AnkJi,AnkOst}. Recently, Anker and Trojan \cite{AnkTro} proved such estimates on the plane in the dihedral case, see also a paper by Graczyk and Sawyer on $A_2$ case \cite{GraSaw:A_2}. 
 
It seems to be slightly more reachable to study the $W$-invariant parts of the Dunkl Laplacians and the associated heat kernels. Graczyk and Sawyer obtained estimates for the rational $W$-invariant heat kernel in the $A_n$ case  \cite{GraSaw:W-inv_D}. In the trigonometric Dunkl setting much less is known. Schapira \cite[Thm. 5.2]{Scha} proved sharp bounds for the $W$-invariant heat kernel at the point $(x,0)$. Since the kernel is not a kernel of convolution operator, this is only a partial result, as Schapira directly underlines in his work.

In this paper we prove genuinely sharp both-sided estimates for the $W$-invariant trigonometric Dunkl kernel $\hOC_t(x,t)$ in rank $1$. Unlike in the rational case, in trigonometric setting there are two distinct crystallographic root systems of rank $1$: $\{-1,1 \}$ and $\{-2,-1,1,2\}$ called $A_1$ and $BC_1$, respectively.

Our main result, stated for the root system $BC_1$, is the following. 

\begin{mainthm}\label{mainthm}
	Let $k_1,k_2\geq 0$ be such that $k_1+k_2>0$. The $W$-invariant Opdam--Cherednik heat kernel satisfies the bound
	\begin{equation*}
		\hOC_t(x,y)\simeq t^{-1/2} e^{-\frac{(x-y)^2}{4t}} e^{-\rho^2 t} e^{-\rho(x+y)} (1+x)(1+y) \frac{(t+1+x+y+xy)^{k_1+k_2-1}}{(t+xy)^{k_1+k_2}},
	\end{equation*}
	uniformly in $t>0$ and $x,y>0$, where $\rho=k_1/2+k_2$.
\end{mainthm}

Here the multiplicity $k_1$ corresponds to the positive root $\al=1$, and $k_2$ corresponds to $\al=2$. By choosing $k_2=0$ one reduces the situation to the $A_1$ case. Similarly, taking $k_1=0$ yields a modified $A_1$ system $\{-2,2\}$, cf. \cite{GraSaw25}.

The trigonometric Dunkl setting in rank $1$ with multiplicities $k_1,k_2$ as above, overlaps with the non-compact Jacobi setup with the type parameters $\al=k_1+k_2-1/2$, $\beta=k_2-1/2$, see Section \ref{S:Jacobi} for more details. A direct consequence of Theorem \ref{mainthm} is the sharp bound for the non-compact Jacobi heat kernel.

\begin{mainthm}\label{thm:Jacobi}
	Let $\al\geq\beta\geq -1/2$, $\al>-1/2$. The non-compact Jacobi heat kernel satisfies  the bound
	\begin{multline*}
		\mathcal{H}_t^{\al,\beta}(x,y)\simeq t^{-1/2} e^{-\frac{(x-y)^2}{4t}} e^{-(\al+\beta+1)^2 t} e^{-(\al+\beta+1)(x+y)}\\
		\times (1+x)(1+y) \frac{(t+1+x+y+xy)^{\al-1/2}}{(t+xy)^{\al+1/2}},
	\end{multline*}
	uniformly in $t>0$ and $x,y>0$.
\end{mainthm}

\subsection{Organization of the paper}

In Section \ref{S:prelimi} we recall the terminology used in the Dunkl theory, both rational and trigonometric. In Section \ref{S:rank1} we state certain preparatory results and introduce auxiliary notation. This part is restricted to rank $1$ case. Finally, in Section \ref{Section:proof} we prove Theorem \ref{mainthm}. We split the proof into four parts: small time, the diagonal, lower bound for large time, and upper bound for large time.

\subsection{Relation to \texorpdfstring{\cite{GraSaw25}}{}}
The sharp Opdam--Cherednik heat kernel estimates for the root system $A_1$ were studied by Graczyk and Sawyer in \cite{GraSaw25}. Unfortunately, the author of the present paper noticed certain inconsistencies in the paper, especially, that the obtained bound does not hold in the stated form. After much discussions with P. Graczyk during his visit in Wrocław in 2025 and unsuccessful attempts to fill the gaps, it was suggested by P. Graczyk that the author could attempt to prove the sharp estimates from scratch. 

The aim of this paper is to establish a rigorous, independent proof of the sharp two-sided estimates, while also extending the result to the broader, non-reduced setting of the $BC_1$ root system. In the Jacobi context, this means considering the full range of the type parameters $\al\geq\be\geq -1/2$, $\al>-1/2$, whereas the system $A_1$ corresponds to $\al=\be>-1/2$ (or, equivalently, to $\al>-1/2$, $\be=-1/2$).

For the completeness of the paper we point out the discrepancies in \cite{GraSaw25}:
\begin{enumerate}
	\item The function $u(r,s,t)$ defined on p. 194 does not satisfy $u(r,s,0)=0$, as claimed. Consequently, the bound in Region 0 does not stand in the stated form. Later on, in Section 3.6, the argument for the bound in the remaining regions relies on the one in Region 0, making the result incorrect both for small and large times.
	\item In Section 3.6 the boundary condition $h^+_{\rm ABCD}(r,r,t) - h_{\rm Opdam}(r,r,t)\geq 0$ was omitted. The symmetry mentioned in Section 3.1 can be only used after completing the bound for $r\geq s$.
	\item When gluing bounds between regions the argument in \cite{GraSaw25} is to check if the error term in the expressions for $D h$ (where $h$ is one of the various glued kernels) is $O(t^{-2})$, which holds readily because of the exponential decay in $t$. However, what should be checked is actually the error term for $D(h)/h$, which is much more involving. See for instance p. 202 and the bound for $D (h_{\rm D}^{\pm})$: if the correct condition is verified the error term is bounded by a constant, not by a multiplicity of $t^{-2}$.  
\end{enumerate}

\subsection{Notation}
We write $f\lesssim g$ is there exists $C>0$ independent of key parameters such that $f\leq C g$. If $f\lesssim g$ and $g\lesssim f$, then we write $f\simeq g$. In the paper all constants may depend on the multiplicities $k_1,k_2>0$. See also Section \ref{section:aux_not} for further notation.

\subsection*{Acknowledgments}
The author is grateful to Piotr Graczyk for introducing him to the trigonometric Dunkl setting, for many conversations on the subject, and for his kind hospitality during author's visit in Angers. Moreover, the author expresses many thanks to Angela Pasquale for sharing her knowledge about the spherical functions.

\section{Preliminaries}\label{S:prelimi}

In this section, we introduce the terminology and notation used in the Dunkl settings. For a general discussion concerning the trigonometric Dunkl setup we refer to original papers by Heckman, Cherednik and Opdam \cite{Che91,Che94,Hec90,Hec91,HecOpd,Opd96} and to the lecture notes by Opdam \cite{Opd_lec} and by Anker \cite{Ank_lec}. 

Let $\Eucl$ be an Euclidean space of dimension $d$ with inner product $\langle \cdot,\cdot\rangle$. A root system $\Sigma\subset \Eucl$ is a finite set of vectors such that  
\begin{itemize}
	\item $\Sigma$ spans $\Eucl$,
	\item for each $\al\in\Sigma$ there holds $\sigma_\al(\Sigma)=\Sigma$, where $\sigma_\al$ is the orthogonal reflection across the hyperplane perpendicular to $\al$, given by the formula
	\begin{equation*}
		\sigma_\al(x)= x - \frac{\langle x,\alpha\rangle}{\langle \al,\al\rangle} \al,
	\end{equation*}
	\item for each $\al,\beta\in\Sigma$ it holds that $\frac{2\langle\al,\beta\rangle}{\langle \al,\al\rangle}\in\ZZ$ (such root systems are called crystallographic).
\end{itemize}
We say that a root system is reduced if $\Sigma\cap\al\RR=\{\al,-\al\}$ for every $\al\in\Sigma$. If a crystallograhic root system is not reduced it may occur that $\al\in\Sigma$ and $2\al\in\Sigma$. The Weyl group $W$ is the subgroup of the orthogonal group $\mathcal{O}(\Eucl)$ generated by $\sigma_\al$, $\al\in\Sigma$. Moreover, let a $W$-invariant function $k\colon \Sigma\to [0,\infty)$ be the multiplicity function. The rank of $\Sigma$ is the dimension of $E$.

Let $v\in \Eucl$ be such that $\al(v)\tc= \langle\al,v\rangle\neq 0$ for $\al\in\Sigma$. We call $\al\in\Sigma$ a positve root if $\al(v)>0$. The set of all positive roots is denoted by $\Sigma_+$. The positive Weyl chamber is defined by
\begin{equation*}
	C_+=\big\{x\in E\colon \al(x)>0\ \forall\al\in\Sigma_+  \big\}
\end{equation*}

\subsection{Trigonometric Dunkl setting}

Let $k\colon\Sigma\to[0,\infty)$. The Dunkl--Cherednik operator $T_\xi$, $\xi\in\Eucl$, is defined on $C^1(\Eucl)$ by
\begin{equation*}
	T_\xi f(x) = \partial_\xi f(x) + \sum_{\al\in\Sigma_+}  k_\al \al(\xi) \frac{  f(x)-f(\sigma_\al(x))}{1-e^{-\al(x)}} - \rho(\xi)f(x),
\end{equation*}
where $\rho=\sum_{\al\in\Sigma_+} \frac{k_\al}{2} \al$. The associated measure is
\begin{equation*}
	\dd\mu(x) = \prod_{\al\in\Sigma_+} \Big| 2\sinh\frac{\al(x)}{2}\Big|^{2k_\al}.
\end{equation*}
Remarkably, the Dunkl--Cherednik operators commute, but are not skew-adjoint. The Opdam--Cherednik Laplacian is given by
\begin{equation*}
	\Delta^{\rm OC} f(x) = \Delta f(x) + \sum_{\al\in\Sigma_+} k_\al \coth\frac{\al(x)}{2} \partial_\al f(x) + |\rho|^2f(x) - \sum_{\al\in\Sigma_+} \frac{k_\al |\al|^2}{4\sinh^2\frac{\al(x)}{2}},
\end{equation*}
where $\Delta$ denotes the classical Laplacian in $\RR^d$. We shall consider the modified version of the Opdam--Cherednik Laplacian, that is
\begin{equation*}
	\widetilde{\Delta}^{\rm OC}_k = \Delta_k^{\rm OC} - |\rho|^2.
\end{equation*}
This modification is irrelevant for the heat kernel estimates (up to a time-depending term),  and thanks to it the associated heat kernel is probabilistic (see \cite[Corollary~5.1]{Scha}). 

The main operator we focus in this paper is the $W$-ivariant part of $\widetilde{\Delta}_k^{\rm OC}$, that is
\begin{equation*}
	L=\Delta + \sum_{\al\in\Sigma_+} k_\al \coth\frac{\al(x)}{2} \partial_\al.
\end{equation*}
The Opdam--Cherednik heat operator is given by
\begin{equation*}
	D=\partial_t - L.
\end{equation*}

The heat kernel $\hOC_t(x,y)$ is the fundamental solution of the Cauchy problem for $D$, see \cite[Section~5.]{Scha}, defined by
\begin{equation}\label{eq:hOC_def}
	\hOC_t(x,y)= e^{-|\rho|^2 t}\int_0^\infty e^{-\lambda^2 t} \varphi_{i\lambda}(x) \varphi_{i\lambda}(-y)\, \dd v'( \lambda),
\end{equation}
where $v'$ is the symmetric Plancherel measure, and $\{\varphi_\lambda\}_{\lambda\in \Eucl}$ denotes the spherical function. The latter is defined as the unique solution of
\begin{equation*}
	\left\{  \begin{array}{rcl}
		p(T) \varphi_\lambda &= & p(\lambda) \varphi_\lambda\\
		\varphi_\lambda(0)&=&1,
	\end{array}    \right.
\end{equation*}
for all $W$-invariant polynomials $p$. Here $p(T)$ for $p(\lambda)=\sum a_\gamma \lambda^{\gamma}$, $\gamma\in\NN^d$, denotes $\sum a_\lambda T_{e_1}^{\gamma_1}\circ\ldots\circ T_{e_d}^{\gamma_d}$. In particular, we have
\begin{equation}\label{eq:7}
	L\varphi_\lambda(x) = \big(|\lambda|^2 +|\rho|^2\big)\varphi_\lambda(x).
\end{equation}

Sharp bounds for the spherical function $\varphi_0$ are known (see \cite[Theorem~3.1]{Scha})
\begin{equation}\label{eq:varphi_0}
	\varphi_0(x)\simeq e^{-\rho(x)} \prod_{\al\in\Sigma_{++}} \big( 1+ \al(x)\big),
\end{equation}
where $\Sigma_{++}$ denotes the set of positive indivisible roots. 

Schapira \cite{Scha} proved that $\hOC_t$ is strictly positive. Moreover, he obtained sharp bounds for the heat kernel with one variable equal to $0$ (mind the change of time $t\mapsto t/2$):
\begin{equation*}
	\hOC_t(x,0) \simeq t^{-\frac{d}{2} -\gamma} e^{-\frac{|x|^2}{4t}} e^{-|\rho|^2 t} e^{-\rho(x)} \prod_{\al\in\Sigma_{++}} \big( 1+\al(x)\big) \big(1+t+\al(x)\big)^{k_\al +k_{2\al}-1},
\end{equation*}
where $\gamma=\sum_{\al\in\Sigma_+} k_\al $. Here and later we use the convention that $k_{2\al}=0$ if $2\al\notin\Sigma$.

\subsection{Rational Dunkl setting}

We briefly recall certain basics of the rational Dunkl setting in rank one. The Dunkl operator is given by
\begin{equation*}
	T_\xi f(x) = \partial_\xi f(x) +\sum_{\al\in\Sigma_+} k_\al \frac{\al(\xi) \big(f(x)-f(\sigma_\al(x)) \big)}{\al(x)}.
\end{equation*}
It is skew-adjoint in $L^2(\RR,\dd\mu^{\rm D})$, where
\begin{equation*}
	\dd\mu^{\rm D}(x) = \prod_{\al\in\sigma_+} |\al(x)|^{2k_\al}\,\dd x.
\end{equation*}
The Dunkl Laplacian is defined as
\begin{equation*}
	\Delta^{\rm D} f(x) = \Delta f(x) +\sum_{\al\in\Sigma_+} \frac{2k_\al}{\al(x)}\partial_\al f(x) -\sum_{\al\in\Sigma_+} \frac{k_\al |\al|^2 \big( f(x)-f(\sigma_\al(x))\big)}{\al(x)^2} 
\end{equation*}
The $W$-invariant Dunkl Laplacian is given by
\begin{equation*}
	L^{\rm D} =  \Delta f(x) +\sum_{\al\in\Sigma_+} \frac{2k_\al}{\al(x)}\partial_\al f(x),
\end{equation*}
and the associated heat operator by
\begin{equation*}
	D^{\rm D} = \partial_t -L^{\rm D}.
\end{equation*}

We denote by $\HDu_t(x,y)$ the heat kernel associated with $L^{\rm D}$. Notably, $\HDu_t$ is the fundamental solution of the heat equation. The sharp estimates for $\HDu$ are known in for the root systems $A_n$ in $\RR^n$ (see \cite{GraSaw23})
\begin{equation}\label{eq:rDunkl_est}
	\HDu_t(x,y)\simeq t^{-n/2} e^{-\frac{(x-y)^2}{4t}}  \prod_{\al\in\Sigma_+} \frac{1}{\big(t+\al(x)\al(y)\big)^k}.
\end{equation}
Recall that for $A_n$ the multiplicicty function is constant.

It is worth mentioning that also non $W$-invariant heat Dunkl kernels were studied, see for instance the paper by Anker and Trojan \cite{AnkTro}, where they consider the dihedral case.

We shall make use of the parabolic minimum principle for unbounded domains, cf. \cite{GraSaw25}, for both $D$ and $D^{\rm D}$.

\begin{lm}(Parabolic minimum principle)\label{lm:PMP}
	Let $T=D$ or $T=D^{\rm D}$. Let $0\le t_1<t_2<\infty$ and $\Omega$ be a connected opensubset of $\RR^d$. Assume that $f \in \mathcal{C}^{2,1}(\Omega\times (t_1,t_2))\cap \mathcal{C}(\overline{\Omega}\times[t_1,t_2])$ satisfies 
	\begin{enumerate}
		\item $Tf(x,t)\geq 0$ for  $(x,t)\in \Omega\times(t_1,t_2]$,
		\item $f(x,t)\geq 0$ for $(x,t)\in\partial\Omega\times [t_1,t_2]\cup \Omega\times \{t_1\}$.
	\end{enumerate}
	Then, $f(x,t)\geq 0$ for $(x,t)\in \overline{\Omega}\times [t_1,t_2]$. 
\end{lm}

\section{Rank 1 situation}\label{S:rank1}

From now on we focus on the general rank one root system $\Sigma=\{1,-1,2,-2\}$. Let $k_1,k_2$ be the multiplicities associated with the roots $\al_1=1$, $\al_2=2$, respectively. We assume that $k_1,k_2\geq 0$, $k_1+k_2>0$. The associated measure on $\RR$ is given by
\begin{equation*}
	\dd\mu(x) = \Big| 2\sinh \frac{x}{2}\Big|^{2k_1} \big| 2\sinh x\big|^{2k_2}  \,\dd x.
\end{equation*}
The $W$-invariant (modified) Opdam--Cherednik Laplacian is
\begin{equation}\label{eq:L_def}
	L= \partial^2_x + \Big(k_1\coth\frac{x}{2}+ 2k_2\coth(x)\Big)\partial_x.
\end{equation}
The associated heat operator $D=\partial_t-L$ satisfies for sufficiently smooth functions the identity
\begin{equation}\label{eq:D(product)}
	D(fg)(x,t) = Df(x,t) g(x,t)+f(x,t) Dg(x,t) +2\partial_x f(x,t)\partial_x g(x,t).
\end{equation}

The formula \eqref{eq:hOC_def} boils down to
\begin{equation}\label{eq:hOC_def_r1}
	\hOC_t(x,y) = e^{-|\rho|^2t} \int_0^\infty e^{-\lambda^2 t}\varphi_{i\lambda}(x) \varphi_{i\lambda}(y) \frac{\dd\lambda}{|c(\lambda)|^2},
\end{equation}
where $\rho=k_1/2+k_2$ and $c(\lambda)$ is the Harish-Chandra function given by
\begin{equation*}
	c(\lambda) = \frac{4^{k_1+k_2}\, \Gamma(k_1+k_2+1/2)\, \Gamma(2\lambda i)\, \Gamma(\lambda i +k_1/2)}{2\sqrt{\pi}\, \Gamma(2\lambda i +k_1)\, \Gamma(\lambda i +k_1/2 +k_2)}.
\end{equation*}
By the classical properties of the Gamma function we obtain for a certain $N\in\NN$ the bound
\begin{equation}\label{eq:c_bound}
	\frac{1}{|c(\lambda)|} \lesssim \lambda (1+\lambda)^N,\qquad \lambda>0.
\end{equation}

Schapira's bound for the $\hOC_t(x,0)$ reduces to
\begin{equation}\label{eq:heat(x,0)}
	\hOC_t(x,0)\simeq t^{-k_1-k_2-1/2} e^{-|\rho|^2 t} e^{-\rho x} e^{-\frac{x^2}{4t}} (1+x) (1+t+x)^{k_1+k_2-1},\quad x>0,\ t>0.
\end{equation}
Clearly, this estimate agrees with Theorem \ref{mainthm}.

By applying the bound \eqref{eq:varphi_0} for the spherical function $\varphi_0$, Theorem \ref{mainthm} can be equivalently formulated as
\begin{equation*}
	\hOC(x,y,t)\simeq t^{-1/2} e^{-\frac{(x-y)^2}{4t}} e^{-\rho^2 t} \varphi_0(x) \varphi_0(y) \frac{(t+1+x+y+xy)^{k_1+k_2-1}}{(t+xy)^{k_1+k_2}}.
\end{equation*}
 
We shall make use of \cite[Theorem~3.3]{Scha} (see also  \cite[Lemma~1]{GraSaw25} and \cite[pp.~9--10]{AnkOst}). Denote
\begin{equation*}
	G(x) = x\Big(\frac{\varphi'_0(x)}{\varphi_0(x)} +\frac{k_1}{2}\coth\frac{x}{2}+ k_2\coth(x)\Big),\qquad x\geq 0.
\end{equation*}

\begin{lm}\label{lm:G}
	For any $x\geq 0$ it holds that $G(x)>0$. Moreover, there exists $K>0$ such that for all
	\begin{equation*}
		\big|G(x)-1\big|\leq \frac{K }{x+1}, \qquad x\geq 0.
	\end{equation*}
\end{lm}
\begin{proof}
	Denote
	\begin{equation*}
		\widetilde{G}(x) = x\Big( \frac{\varphi'_0(x)}{\varphi_0(x)} +\frac{k_1}{2} + k_2\Big)
	\end{equation*}
	Directly from \cite[Theorem~3.3~2.]{Scha} we have 
	\begin{equation*}
		\widetilde{G}(x)\geq 0\qquad \text{and}\qquad \big| \widetilde{G}(x)-1\big|\lesssim \frac{1}{1+x}.
	\end{equation*}
	It suffices to use the facts $G(x)>\widetilde{G}(x)$ and that $x\mapsto x(x+1)(\coth x-1)$ is a bounded function. Mind that $G(0)=k_1+k_2>0$.
\end{proof}

\begin{rem}
	Notice that Lemma \ref{lm:G} implies that
	\begin{equation*}
		\inf_{x>0} \Big(2G(x)+\frac{1}{1+x^2}\Big)>0.
	\end{equation*}
	We define
	\begin{equation}\label{eq:CG}
		\CG=\frac{K}{\inf_{x>0} (2G(x)+\frac{1}{1+x^2})},
	\end{equation}
	where $K$ is as in Lemma \ref{lm:G}. 
\end{rem}

\subsection{Non-compact Jacobi setting}\label{S:Jacobi}
Let $\al\geq\beta\geq-1/2$, $\al>-1/2$. The Jacobi Laplacian $\Delta^{\al,\beta}$ is given by
\begin{equation*}
	\Delta^{\al,\be}  = -\partial_x^2 - \Big((2\al+1)\coth x + (2\be+1)\tanh x\Big)\partial_x.  
\end{equation*}
The Jacobi function $\varphi_\lambda^{\al,\be}$ is the unique smooth even solution of the equation
\begin{equation*}
	\Delta^{\al,\be} u(x) = \big( |\lambda|^2 + (\al+\be+1)^2\big) u(x),\qquad u(0)=1.
\end{equation*}
Moreover, for every $x\in\RR$ the function $\lambda\mapsto \varphi_\lambda^{\al,\be}(x)$ is analytic and even. The Jacobi functions form an orthogonal system on $((0,\infty),\dd\mu_{\al,\be})$, where
\begin{equation*}
	\dd \mu_{\al,\be}(x) = \big(2\sinh x\big)^{2\al+1} \big(2\cosh x\big)^{2\be+1} \dd x. 
\end{equation*}

The associated heat kernel $\mathcal{H}_t^{\al,\be}$ is the fundamental solution of the system
\begin{equation*}
	\left\{ \begin{array}{rl}
		\partial_t u(x,t)&= -\Delta^{\al,\be} u(x,t),\\
		\lim\limits_{t\to 0} u(x,t)&=f(x).
	\end{array}\right.
\end{equation*}
Sharp estimates for $\mathcal{H}_t^{\al,\be}(x,0)$ analogous to \eqref{eq:heat(x,0)} were shown in \cite{KawLiu}. For the full discussion of the Jacobi setting we refer to Koornwinder \cite{Koornwinder_book} (see also \cite{AstBla}).

The Opdam--Cherednik setup in rank one is closely related with the non-compact Jacobi setting. In particular, the operator $L$ is equivalent to the Jacobi operator for the type parameters $\alpha=k_1+k_2-1/2$ and $\beta=k_2-1/2$. More precisely, 
\begin{equation*}
	MLM^{-1}f(x) = -\frac{1}{4} \Delta^{k_1+k_2,-1/2,k_2-1/2} f(x),
\end{equation*}
where $Mf(x)=f(2x)$. 

Consequently, the spherical functions $\varphi_\lambda$ correspond to the Jacobi functions via (see \cite[Remark~4.6]{Ank_lec})
\begin{equation*}
	\varphi_\lambda(2x) = \varphi_{2i\lambda}^{k_1+k_2-1/2,k_2-1/2}(x),\qquad x>0,\ \lambda\in\CC.
\end{equation*}
Moreover, the following relation between the heat kernels is valid
\begin{equation}\label{eq:OC-Jacobi}
	\mathcal{H}_t^{k_1+k_2-1/2,k_2-1/2}(x,y) = 4 \hOC_{4t}(2x,2y).
\end{equation}
Therefore, Theorem \ref{thm:Jacobi} is equivalent to Theorem \ref{mainthm}.

The following formula is valid
\begin{equation}\label{eq:hOC-Jacobi}
	\hOC_t(x,y) = 2\int_0^\infty \hOC_t(z,0) W\Big(\frac{x}{2},\frac{y}{2},\frac{z}{2}\Big) \dd \mu(z), 
\end{equation}
where $W(x,y,z)$ is non-negative, even in $z$, supported in $|z|\in[|x-y|,x+y]$, and satisfying
\begin{equation*}
	2\int_0^\infty W\Big(x,y,\frac{z}{2}\Big)  \dd\mu(z)=1, \qquad x,y>0,
\end{equation*} 
see \cite[pp.~58--59]{Koornwinder_book}. Moreover, we shall make use of the product formula for the spherical functions:
\begin{equation}\label{eq:phi_prod}
	\varphi_0(x)\varphi_0(y)=\int_0^\infty \varphi_0(z) W\Big(\frac{x}{2},\frac{y}{2},\frac{z}{2}\Big) \dd \mu(z).
\end{equation}

\subsection{Auxiliary notation}\label{section:aux_not}

For a function $\psi(x,y,t)$ denote
\begin{equation*}
	H^\psi_t(x,y)= t^{-1/2} e^{-\frac{(x-y)^2}{4t}} e^{-\rho^2 t} \varphi_0(x) \varphi_0(y) \psi(x,y,t).
\end{equation*}
Theorem \ref{mainthm} states that $\hOC_t (x,y)\simeq H^{\psi_0}_t(x,y)$, where 
\begin{equation}\label{eq:psi_0}
	\psi_0(x,y,t)=\frac{(t+1+x+y+xy)^{k_1+k_2-1}}{(t+xy)^{k_1+k_2}}.
\end{equation}
We shall simply write $H^0_t:=H^{\psi_0}_t$.

We use the following notation. For a function $f$ defined on $\RR^2\times(0,\infty)$ we abbreviate $D\big(f(\cdot,y,\cdot)\big)(x,t)$ to  $Df(x,y,t)$. The operator $D$ always acts on the variables $x$ and $t$. Moreover, it will be convenient to use the symbol $\eth$ to denote the following operation
\begin{equation*}
	\eth_x^n f(x,y,r)=\frac{\partial_x^n f(x,y,t)}{f(x,y,t)},\qquad n\in\NN,
\end{equation*}
and similarly for $\eth_t$.

We have the following lemma.

\begin{lm}\label{lm:gen_form}
	Let $\psi\in C^2((0,\infty)^3)$. The following formula is valid
	\begin{multline*}
		\frac{D H^\psi_t(x,y)}{H^\psi_t(x,y)}\\
		= \big( G(x)-1\big) \Big( \frac{x-y}{xt} -\frac{2 \eth_x \psi}{x}\Big) +\eth_t\psi +\frac{x-y}{xt} +\Big(\frac{x-y}{t} -\frac{2}{x}\Big)\eth_x \psi -\eth^2_x\psi.
	\end{multline*}
\end{lm}
\begin{proof}
	Notice that \eqref{eq:7} gives $L \varphi_0(x)= \rho^2\varphi_0$. We use the one-off notation
	\begin{equation*}
		h^0_t(x,y)= t^{-1/2} e^{-\frac{(x-y)^2}{4t}} e^{-\rho^2 t} \varphi_0(x) \varphi_0(y).
	\end{equation*}
	By \eqref{eq:D(product)} we have
	\begin{align*}
		\frac{D h^0_t(x,y)}{h^0_t(x,y)}&=-\frac{1}{2t} +\frac{(x-y)^2}{4 t^2} -\rho^2 +\frac{x-y}{t}\Big( \frac{k_1}{2}\coth\frac{x}{2}+k_2\coth x\Big)  -\frac{(x-y)^2}{4t^2} \\
		&\qquad+\frac{1}{2t}+\rho^2 +\frac{x-y}{t} \cdot\frac{\varphi'_0(x)}{\varphi_0(x)}\\
		&=\frac{(x-y)G(x)}{xt}.
	\end{align*}
	Now we again apply \eqref{eq:D(product)} and obtain
	\begin{align*}
		\frac{D H^\psi_t(x,y)}{H^\psi_t(x,y)}
		& = \frac{(x-y)G(x)}{xt} +\eth_t \psi -\Big(k_1\coth\frac{x}{2}+ 2k_2\coth x\Big)\eth_x \psi - \eth_x^2\psi \\
		&\qquad	+2\eth_x\psi \Big(\frac{x-y}{2t}-\frac{\varphi'_0(x)}{\varphi_0(x)}\Big)\\
		&=G(x)\Big(\frac{x-y}{xt}-\frac{2\eth_x\psi}{x}\Big)+\eth_t\psi +\frac{(x-y)\eth_x \psi}{t} -\eth^2_x\psi.
	\end{align*}
	A simple rearrangement of the terms concludes the proof. 
\end{proof}

\section{Proof of Theorem \ref{mainthm}}\label{Section:proof}

\subsection{Small time}

In this section we prove Theorem \ref{mainthm} for $t\leq 1$. We denote
\begin{equation*}
	\delta(x) = \Big(2\sinh\frac{x}{2}\Big)^{2k_1} \Big( 2 \sinh x\Big)^{2k_2}, 
\end{equation*}
which is the density of the measure $\dd\mu$. Moreover, we also define its counterpart corresponding to the rational Dunkl setting of rank one and the multiplicity $k=k_1+k_2$:
\begin{equation*}
	\delta_{\rm D} (x) = (2 x)^{2k_1+2k_2}.
\end{equation*}

\begin{lm}\label{lm:L,L_Dunkl}
	If $f\in \mathcal{C}^2((0,\infty))$, then
	\begin{equation*}
		\frac{L\Big( \big(\frac{\delta_{\rm D}(x)}{\delta(x)}\big)^{1/2} f(x)\Big) }{ \big(\frac{\delta_{\rm D}(x)}{\delta(x)}\big)^{1/2}  f(x)} = \frac{L^{\rm D}f(x)}{f(x)} -\rho^2+R_{k_1,k_2}(x),
	\end{equation*}
	where $L^D$ is the Dunkl Laplacian corresponding to the multiplicity $k=k_1+k_2$, and
	\begin{equation*}
		\big|R_{k_1,k_2}(x)\big| \lesssim 1,\qquad x>0.
	\end{equation*}
\end{lm}
\begin{proof}
Denote $g(x)=\big(\frac{\delta_{\rm D}(x)}{\delta(x)}\big)^{1/2}$. By the definition \eqref{eq:L_def} of $L$ we have
\begin{equation*}
	\frac{L(g f)}{gf} = \eth^2f  + \eth f\big( 2\eth g' +k_1\coth\frac{x}{2} +2k_2\coth(x) \big) + \eth^2g+k_1 \eth g \coth\frac{x}{2}+2k_2\eth g\coth x.  
\end{equation*}

Observe that
\begin{equation*}
	\eth_x g(x)=\frac{k_1}{2}\Big(\frac{2}{x}-\coth\frac{x}{2}\Big) + k_2\Big(\frac{1}{x}-\coth x\Big).
\end{equation*}
Moreover,
	\begin{multline*}
		\eth_x^2 g(x) = \Big[\frac{k_1}{2}\Big(\frac{2}{x}-\coth\frac{x}{2}\Big) + k_2\Big(\frac{1}{x}-\coth x\Big)\Big]^2\\
		 +\Big[\frac{k_1}{2}\Big(-\frac{2}{x}+\frac{1}{2\sinh^2\frac{x}{2}}\Big) + k_2\Big(-\frac{1}{x^2}+\frac{1}{\sinh^2 x}\Big)\Big].
	\end{multline*}
		
	We directly compute
	\begin{equation*}
		2\eth g' +k_1\coth\frac{x}{2} +2k_2\coth(x) = \frac{2(k_1+k_2)}{x}
	\end{equation*}
	and
	\begin{multline*}
		\eth^2g+k_1 \eth g \coth\frac{x}{2}+2k_2\eth g\coth x = -\rho^2+\frac{k_1^2-k_1}{4}\Big[ \Big(\frac{2}{x}\Big)^2 -\frac{1}{\sinh^2\frac{x}{2}}\Big]\\
		 + (k_2^2-k_2)\Big[\frac{1}{x^2} - \frac{1}{\sinh^2 x}\Big] +k_1k_2\Big[ 1+\frac{2}{x^2} -\coth\frac{x}{2} \coth x\Big]=: -\rho^2+R_{k_1,k_2}(x). 
	\end{multline*}
	
	By combining the above we arrive at
	\begin{equation*}
		\frac{L(fg)}{fg} = \frac{f'' + \frac{2(k_1+k_2)}{x} f'}{f} -\rho^2+R_{k_1,k_2}(x).
	\end{equation*}
	It is straightforward to verify that $R_{k_1,k_2}(x)$ is bounded.
	
\end{proof}

\begin{prop}\label{prop:small_t}
	It holds that
	\begin{equation*}
		\hOC_t(x,y)\simeq H^0_t(x,y),
	\end{equation*}
	uniformly in $x,y>0$ and $t\in(0,1]$.
\end{prop}
\begin{proof}
	Fix $y>0$. Let us denote
	\begin{equation*}
		H_t(x,y) = \HDu_t(x,y) e^{-\rho^2 t} \Big(\frac{\delta_{\rm D}(x)\delta_{\rm D}(y)}{\delta(x)\delta(y)}\Big)^{1/2},
	\end{equation*}
	where the rational Dunkl kernel $\HDu_t$ is associated with the root system $A_1$ and the multiplicity $k_1+k_2$.
	Notice that \eqref{eq:rDunkl_est}, \eqref{eq:varphi_0}, and the bound $x/\sinh x\simeq (1+x)e^{-x}$, $x>0$, imply
	\begin{equation}\label{eq:6}
		H_t(x,y)\simeq H^0_t(x,y),\qquad x,y>0,\ t\in(0,1].
	\end{equation}
	Therefore, our task is reduced to showing $H_t(x,y)\simeq \hOC_t(x,y)$.
	
	By Lemma \ref{lm:L,L_Dunkl} we have (recall that $L$, $L^{\rm D}$ act on $x$ and $D$, $D^{\rm D}$ act on $x$ and $t$)
	\begin{equation*}
	\begin{split}
		\frac{D H_t(x,y)}{H_t(x,y)} = -\rho^2 +\frac{\partial_t \HDu_t(x,y)}{\HDu_t(x,y)} - \frac{L\big( \big(\frac{\delta_{\rm D}(x)}{\delta(x)}\big)^{1/2}\, \HDu_t(x,y)\big)}{\big(\frac{\delta_{\rm D}(x)}{\delta(x)}\big)^{1/2}\, \HDu_t(x,y)}= -R_{k_1,k_2}(x).
	\end{split}
	\end{equation*}
	Here we used the fact that $\HDu_t$ is the fundamental solution of the heat equation for $D^{\rm D}$. Hence,
	\begin{equation*}
		\Big|\frac{D H_t(x,y)}{H_t(x,y)}\Big|\lesssim 1,\qquad x>0.
	\end{equation*}
	
	We shall apply Lemma \ref{lm:PMP} to the function
	\begin{equation*}
		u(x,y,t) = -\hOC_t(x,y)+e^{Ct} H_t(x,y)
	\end{equation*}
	in $(x,t)\in(0,\infty)\times(0,1)$, where $C$ is a large positive constant. We verify that $u$ satisfies the required conditions. Firstly,
	\begin{equation*}
		\frac{D u(x,y,t)}{e^{Ct} H_t(x,y)} = C + \frac{D H_t(x,y)}{H_t(x,y)}\geq 0,
	\end{equation*}
	for $C$ large enough. Secondly, for $t=0$ and a test function $f$ we have
	\begin{multline*}
	\lim_{t\to 0} \int u(x,y,t) f(y) \, \dd \mu(y) =   -\lim_{t\to 0} \int \hOC_t(x,y) f(y) \, \dd \mu(y)\\
	 +  \Big(\frac{\delta_{\rm D}(x)}{\delta(x)}\Big)^{1/2}	\lim_{t\to 0}  \int \HDu_t(x,y) f(y)\Big(\frac{\delta(y)}{\delta_{\rm D}(y)}\Big)^{1/2} \, \dd \mu^{\rm D}(y)=-f(x)+f(x)=0.
	\end{multline*}
	Thus, $u(x,y,0)=0$. Lastly, if $x=0$ then by Schapira's result we have $H_t(0,y)\simeq \hOC_t(0,y)$, independently of $y$. Hence, for $C$ large enough we have $u(0,y,t)\geq 0$. Therefore, by Lemma \ref{lm:PMP}
	\begin{equation*}
		u(x,y,t)\geq 0,\qquad x,y>0,\ t\in(0,1].
	\end{equation*}
	This and \eqref{eq:6} imply $H_t(x,y)\gtrsim \hOC_t(x,y)$.
	
	In order to prove the opposite bound it suffices to show that for $x,y>0$ and $t\in(0,1]$ it holds that 
	\begin{equation*}
		\widetilde{u}(x,y,t):= e^{Ct} e^{\rho^2 t}\Big(\frac{\delta(x)\delta(y)}{\delta_{\rm D}(x)\delta_{\rm D}(y)}\Big)^{1/2 } \hOC_t(x,y) -\HDu_t(x,y)\geq 0.
	\end{equation*}
	Much as above, we use Lemma \ref{lm:PMP} for $\widetilde{u}$ in $(x,t)\in (0,\infty)\times (0,1]$, but this time for the operator $D^{\rm D}$. For that purpose, observe that Lemma \ref{lm:L,L_Dunkl} implies
	\begin{multline*}
			\frac{D^{\rm D}\big[  e^{\rho^2 t} \big(\frac{\delta(x)}{\delta_{\rm D}(x)}\big)^{1/2 }\, \hOC_t(x,y) \big]}{ e^{\rho^2 t} \big(\frac{\delta(x)}{\delta_{\rm D}(x)}\big)^{1/2 }\, \hOC_t(x,y)}\\
			= \rho^2 +\frac{\partial_t \hOC_t(x,y)}{\hOC_t(x,y)}-\frac{L^{\rm D} \big[\big(\frac{\delta(x)}{\delta_{\rm D}(x)}\big)^{1/2 } \hOC_t(x,y) \big] }{\big(\frac{\delta(x)}{\delta_{\rm D}(x)}\big)^{1/2 } \hOC_t(x,y)}=  -R_{k_1,k_2}(x).
	\end{multline*}
	Since $\HDu_t$ is the fundamental solution of the heat equation, for $C$ large enough we have
	\begin{equation*}
		D^{\rm D}\, \widetilde{u}(x,y,t)= D^{\rm D} \Big[ e^{Ct} e^{\rho^2 t} \Big(\frac{\delta(x)\delta(y)}{\delta_{\rm D}(x)\delta_{\rm D}(y)}\Big)^{1/2 }  \hOC_t(x,y) \Big]=C-R_{k_1,k_2}(x)\geq 0.
	\end{equation*}
	Verification of the boundary conditions for  $\widetilde{u}$ is fully analogous to what was shown for $u$.
	
\end{proof}

\subsection{On-diagonal estimate}\label{section:on diagonal}

In this subsection we prove $\hOC_t(x,x)\simeq H^0_t(x,x)$ for all $x>0$. This will allow us to use the parabolic minimum principle in the sequel.

\begin{prop}\label{prop:diagonal}
	We have
	\begin{equation*}
		\hOC_t(x,x)\simeq H^0_t(x,x),
	\end{equation*}
	uniformly in $x>0$ and $t>0$.
\end{prop}
\begin{proof}
	Notice that by Proposition \ref{prop:small_t} we may assume $t\geq 1$. 
	
	Let $a$ be a large positive constant. We shall consider two cases. Firstly, let $x\leq a\sqrt{t}$. Recall that (see \cite{Scha} or \cite{KawLiu})
	\begin{equation*}
		\hOC_t(x,0)\simeq t^{-k_1-k_2-1/2} e^{-\rho^2 t} e^{-x^2/(4t)}\varphi_0(x) (1+t+x)^{k_1+k_2-1} \simeq t^{-3/2} e^{-\rho^2 t} \varphi_0(x),
	\end{equation*}
	where the last bound holds under our the assumptions of the present case. Thus, \eqref{eq:hOC-Jacobi} and \eqref{eq:phi_prod} give
	\begin{equation*}
		\hOC_t(x,x)\simeq t^{-3/2} e^{-\rho^2 t} [\varphi_0(x)]^2,
	\end{equation*}
	and the latter expression is comparable with $H^0_t(x,x)$ for $t\geq 1$ and $x\leq a\sqrt{t}$.
	
	We move on to the second case: $x\geq a\sqrt{t}$. This time we will rely on the following bound for the spherical function, see \cite[(5.1.27)]{GanVar} (mind the convention for the Harish-Chandra function), 
	\begin{equation*}
		\varphi_{i\lambda}(x) = e^{-\rho x}\Big(c(\lambda) e^{i\lambda x} +c(-\lambda) e^{-i\lambda x} + E(\lambda,x)\Big),
	\end{equation*}
	where $E$ is the error part satisfying the bound
	\begin{equation}\label{eq:Error}
		\big| E(\lambda, x)\big| \lesssim (1+\lambda)^N e^{-\varepsilon x}  
	\end{equation}
	for certain $\varepsilon>0$ and $N\in\NN$.
	
	Since $x\mapsto \varphi_\lambda(x)$ is even and $c(-\lambda)=\overline{c(\lambda)}$, the formula \eqref{eq:hOC_def_r1} implies
	\begin{multline*}
		\hOC_t(x,x) = e^{-\rho^2 t} \int_0^\infty e^{-\lambda^2 t} \Big(\frac{\varphi_{i\lambda}(x)}{|c(\lambda)|}\Big)^2\dd \lambda= e^{-\rho^2 t} e^{-2\rho x} \int_0^\infty e^{-\lambda^2 t}\\
		\times\Big[2+ \frac{c(\lambda)e^{2i\lambda x}}{\overline{c(\lambda)}} +  \frac{\overline{c(\lambda)}e^{-2i\lambda x}}{c(\lambda)}+\frac{E(\lambda,x)^2}{|c(\lambda)|^2}  +2E(\lambda,x)\Big(\frac{e^{i\lambda x}}{c(-\lambda)} + \frac{e^{-i\lambda x}}{c(\lambda)}  \Big)\Big]\dd\lambda.
	\end{multline*}
	We shall bound or compute each term of the above expression.
	
	Firstly, we have
	\begin{equation*}
		\int_0^\infty 2 e^{-\lambda^2 t}\dd\lambda = \frac{\sqrt{\pi}}{\sqrt{t}}. 
	\end{equation*}
	Further, \eqref{eq:c_bound} and \eqref{eq:Error} imply there exist $N\in\NN$ and $\varepsilon>0$ such that
	\begin{equation*}
		\int_0^\infty e^{-\lambda^2 t} |E(\lambda,x)|^2 \frac{\dd\lambda}{|c(\lambda)|^2} \lesssim e^{-\varepsilon x}\int_0^\infty e^{-\lambda^2 t} \lambda^2 (1+\lambda)^N \dd\lambda \lesssim e^{-\varepsilon x} t^{-3/2}.
	\end{equation*}
	Much as above,
	\begin{equation*}
		\int_0^\infty e^{-\lambda^2 t} \frac{|E(\lambda,x)|}{|c(\lambda)|}\dd\lambda \lesssim e^{-\varepsilon x} \int_0^\infty e^{-\lambda^2 t} \lambda (1+\lambda)^N \dd\lambda \lesssim e^{-\varepsilon x} t^{-1}.
	\end{equation*}
	Finally, it remains to deal with
	\begin{equation*}
		\int_0^\infty e^{-\lambda^2 t}  \frac{c(\lambda)e^{2i\lambda x}}{\overline{c(\lambda)}} \dd\lambda;
	\end{equation*}
	the other term can be treated in the same manner. For that purpose we consider 
	\begin{equation*}
		F(z)=e^{-z^2 t} \frac{c(z)e^{2i z x}}{\overline{c(z)}}
	\end{equation*}
	in the rectangle ${\rm Re }\,z\in [r,R]$, ${\rm Im}\, z\in [0,x/t]$ in the complex plane, where $R>r>0$. Since $F$ is analytic in the said domain the Cauchy integral theorem gives
	\begin{equation*}
		\int_{r}^R F(s) \dd s =  -\int_0^{x/t} F(R+is)\dd s + \int_0^{x/t} F(r+is)\dd s+\int_r^R F(s+ix/t)\dd s.
	\end{equation*}
	
	We shall bound the right hand side. Observe that for the first term we have
	\begin{equation*}
		\int_0^{x/t} |F(R+is)|\dd s \leq e^{-R^2 t}\int_0^{x/t} e^{s(st-x)} e^{-xs} \dd s\leq \frac{e^{-tR^2}}{x}.
	\end{equation*}
	Much as above,
	\begin{equation*}
		\int_0^{x/t} |F(r+is)|\dd s \leq\frac{e^{-tr^2}}{x}\leq \frac{1}{x}.
	\end{equation*}
	Lastly, for the third term we obtain
	\begin{equation*}
	\int_r^R |F(s+ix/t)|\dd s \leq e^{-x^2/t} \int_0^\infty e^{-s^2 t}\dd s \simeq \frac{e^{-x^2/t}}{\sqrt{t}}. 
	\end{equation*}
	
	Thus,
	\begin{equation*}
		\Big|\int_0^\infty e^{-\lambda^2 t}  \frac{c(\lambda)e^{2i\lambda x}}{\overline{c(\lambda)}} \dd\lambda\Big|\lesssim \frac{1}{\sqrt{t}}\Big(e^{-x^2/t}+\frac{\sqrt{t}}{x}\Big).
	\end{equation*}
	By choosing $a$ large enough in the condition $x\geq a\sqrt{t}$, we arrive at
	\begin{equation*}
		\hOC_t(x,x)\simeq e^{-\rho^2 t} e^{-2\rho x} t^{-1/2},
	\end{equation*}
	which by \eqref{eq:varphi_0} is comparable with $H^0_t(x,x)$ in the considered region.
	
	This concludes the proof of the proposition.
\end{proof}

\subsection{Large time, lower bound}
From now on, whenever it is convenient, we abbreviate $k_1+k_2$ to $k$.

Let $\chi$ be a smooth non-increasing function on $\RR$, such that $\chi(s)=1$ for $s\leq 1$, $\chi(s)=0$ for $s\geq 2$, and $0\le\chi(s)\le 1$ for $s\in[1,2]$. We define 
\begin{equation*}
	H_t^-(x,y) =  C_1^-\Big[1-\chi\Big(\frac{x}{t}\Big)\Big] H^{\psi_1}_t(x,y) + \chi\Big(\frac{x}{t}\Big) H^{\psi_2^-}_t(x,y), 
\end{equation*}
where 
\begin{equation}\label{eq:C_1-}
	C_1^-= e^{\sqrt{5}\CG},
\end{equation}
$\CG$ was defined in \eqref{eq:CG}, and
\begin{align}\label{eq:psi1_psi2-}
		\begin{split}
		\psi_1(x,y,t) &= \frac{(x+xy)^{k-1}}{(t+xy)^k},\\
		\psi_2^-(x,y,t) &= \frac{1}{2t+xy} e^{\frac{\CG \sqrt{x^2+1}}{t} },
		\end{split}
\end{align}

\begin{lm}\label{lm:lowercompar}
	The following bound holds uniformly in $t\geq 1$ and $0<y\le x<\infty$,
	\begin{equation*}
		H^-_t(x,y)\simeq H^0_t(x,y).
	\end{equation*}
\end{lm}
\begin{proof}
	Let $0<y\le x<\infty$ and $t\geq 1$. We show that
	\begin{equation*}
		C_1^-\Big[1-\chi\Big(\frac{x}{t}\Big)\Big] \psi_1(x,y,t) + \chi\Big(\frac{x}{t}\Big) \psi^-_2(x,y,t) \simeq \psi_0(x,y,t),
	\end{equation*}
	where $\psi_0$ is defined in \eqref{eq:psi_0}, with the implicit constants independent of $x,y,t$. Denote the left hand side of the above expression by $\psi$.
	
	In the case $x\leq t$ we have
	\begin{equation*}
	\psi(x,y,t)= \psi_2^-(x,y,t)\simeq \frac{1}{t+xy} \simeq	  \psi_0(x,y,t)  .
	\end{equation*}
	Secondly, if $x\geq 2t$, then
	\begin{equation*}
		\psi(x,y,t)=C_1^-\psi_1(x,y,t)\simeq \psi_0(x,y,t).
	\end{equation*}
	Finally, if  $t\le x\le 2t$, then
	\begin{equation*}
		\psi_1(x,y,t) \simeq \frac{1}{t+xy},\qquad \psi_2^-(x,y,t)\simeq \frac{1}{t+xy}, \qquad \psi_0(x,y,t)\simeq \frac{1}{t+xy}.
	\end{equation*}
	Thus, $\psi(x,y,t)\simeq \psi_0(x,y,t)$.
	
	This concludes the proof.
	
\end{proof}

\begin{lm}\label{lm:Hpsi1}
	One has
	\begin{equation*}
		\Big|\frac{DH^{\psi_1}_t(x,y)}{H^{\psi_1}_t(x,y)}\Big| \lesssim \frac{1}{t^2},	
	\end{equation*}
	uniformly in $t\geq 1$ and $0<y\le x<\infty$ satisfying $t\le x$.
\end{lm}
\begin{proof}
	Firstly, observe that
	\begin{align*}
		\eth_t \psi_1(x,y,t) &= -\frac{k}{t+xy},\\
		\eth_x \psi_1(x,y,t) &= \frac{kt}{x(t+xy)}-\frac{1}{x},\\
		\eth_x^2 \psi_1(x,y,t) &= \frac{1}{x^2} \Big[ \Big(\frac{kt}{t+xy}-1\Big)^2 -\frac{kt(t+2xy)}{(t+xy)^2} +1\Big] = \frac{2}{x^2} +\frac{k^2t^2 -kt(3t+4xy)}{x^2 (t+xy)^2}.
	\end{align*}
	
	By applying Lemma \ref{lm:gen_form} we obtain
	\begin{equation*}
		\begin{split}
			\frac{D H^{\psi_1}_t(x,y)}{H_t^{\psi_1}(x,y)} &= \big( G(x)-1\big) \Big( \frac{x-y}{xt} -\frac{2kt}{x^2(t+xy)} +\frac{2}{x^2}\Big) -\frac{k}{t+xy} +\frac{x-y}{xt}\\
			&\qquad+\Big(\frac{x-y}{t} -\frac{2}{x}\Big)\Big(\frac{kt}{x(t+xy)}-\frac{1}{x}\Big) -\frac{2}{x^2} -\frac{k^2t^2 -kt(3t+4xy)}{x^2 (t+xy)^2}\\
			&=\big( G(x)-1\big) \Big( \frac{x-y}{xt} -\frac{2kt}{x^2(t+xy)} +\frac{2}{x^2}\Big)\\
			&\qquad-\frac{ky^2}{(t+xy)^2}-\frac{k^2 t^2 }{x^2 (t+xy)^2}  +\frac{kt}{x^2 (t+xy)}.
		\end{split}
	\end{equation*}
	
	Thus, the inequality $1\le t\le x$ and Lemma \ref{lm:G} imply that
	\begin{equation*}
		\Big|\frac{D(H_t^{\psi_1}(x,y))}{H_t^{\psi_1}(x,y)}\Big| \lesssim \frac{1}{tx} +\frac{1}{x^2}\lesssim \frac{1}{t^2},
	\end{equation*}
	which finishes the proof.
\end{proof}

\begin{lm}\label{lm:Hpsi2-}
	Let $t\geq 1$ and $0<y\le x<\infty$ be such that $x\le 2t$. The following identity holds:
	\begin{equation*}
		\frac{D H_t^{\psi_2^-}(x,y)}{H_t^{\psi_2^-}(x,y)} = I_1(x,y,t) +I_2(x,y,t),
	\end{equation*}
	where $I_1(x,y,t)=O(t^{-3/2})$ and $I_2(x,y,t)\leq 0$.
\end{lm}
\begin{proof}
By a direct computation we obtain
\begin{align*}
	\eth_t \psi_2^-(x,y,t) &= -\frac{2}{2t+xy} -\frac{\CG \sqrt{x^2+1}}{t^2},\\
	\eth_x \psi_2^-(x,y,t) &= -\frac{y}{2t+xy}+\frac{\CG x}{t\sqrt{x^2+1}},\\
	\eth^2_x \psi_2^-(x,y,t) &= \frac{2y^2}{(2t+xy)^2} +\frac{\CG^2 x^2}{t^2(x^2+1)}+\frac{\CG}{t(x^2+1)^{3/2}} -\frac{2\CG xy}{t(2t+xy)\sqrt{x^2+1}}.
\end{align*}	
Hence, by Lemma \ref{lm:gen_form}
\begin{equation*}
	\begin{split}
	&\hspace{-0.5cm}\frac{D H_t^{\psi_2^-}(x,y)}{H_t^{\psi_2^-}(x,y)}\\
	 &= \big( G(x)-1\big) \Big( \frac{x-y}{xt} +\frac{2y }{x(2t+xy)} -\frac{2\CG}{t\sqrt{x^2+1}}\Big)  -\frac{2}{2t+xy} -\frac{\CG \sqrt{x^2+1}}{t^2}\\
	& \qquad +\frac{x-y}{xt} +\Big(\frac{x-y}{t} -\frac{2}{x}\Big)\Big( -\frac{y}{2t+xy}+\frac{\CG x}{t\sqrt{x^2+1}}\Big)  -\frac{2y^2}{(2t+xy)^2}\\
	&\qquad-\frac{\CG^2 x^2}{t^2(x^2+1)}
	 -\frac{\CG}{t(x^2+1)^{3/2}} +\frac{2\CG xy}{t(2t+xy)\sqrt{x^2+1}}\\
	 &= \big( G(x)-1\big) \Big( \frac{1}{t} -\frac{y^2 }{t(2t+xy)} -\frac{2\CG}{t\sqrt{x^2+1}}\Big)-\frac{2y^2}{(2t+xy)^2}-\frac{2\CG}{t\sqrt{x^2+1}}	 \\
	 &\qquad-\frac{\CG}{t(x^2+1)^{3/2}}-\frac{\CG^2 x^2}{t^2 (x^2+1)} -\frac{\CG}{t^2\sqrt{x^2+1}} -\frac{\CG (xy)^2}{t^2 \sqrt{x^2+1} (2t+xy)}.
	\end{split}
\end{equation*}
	
	We consider two cases. Firstly, assume that $xy\ge t$. Then we take
	\begin{equation*}
		I_1(x,y,t) = \big( G(x)-1\big) \Big( \frac{1}{t} -\frac{y^2 }{t(2t+xy)} -\frac{2\CG}{t\sqrt{x^2+1}}\Big).
	\end{equation*}
	By Lemma \ref{lm:G} this is bounded by a constant multiple of $((x+1)t)^{-1}$. Since $x\geq\sqrt{t}$ we obtain $I_1=O(t^{-3/2})$. The remaining terms are negative and thus included in $I_2$.
	
	Secondly, assume $xy\le t$. We show that in this case $D H_t^{\psi_2^-}(x,y)<0$. For this purpose it suffices to justify that
	\begin{equation}\label{eq:8}
		\big( G(x)-1\big) \Big( 1 -\frac{y^2 }{2t+xy} -\frac{2\CG}{\sqrt{x^2+1}}\Big)-\frac{2\CG}{\sqrt{x^2+1}}	 -\frac{\CG}{(x^2+1)^{3/2}}<0.
	\end{equation}
	The left hand side of the above expression can be rewritten as
	\begin{equation*}
		\big( G(x)-1\big) \Big( 1 -\frac{y^2 }{2t+xy}\Big) -\frac{\CG}{\sqrt{x^2+1}} \Big(2G(x)+\frac{1}{x+1}\Big).
	\end{equation*}
	By Lemma \ref{lm:G} and the definition of $\CG$ \eqref{eq:CG} we see that
	\begin{equation*}
	\big| G(x)-1\big| \Big| 1 -\frac{y^2 }{(2t+xy)}\Big|\leq \frac{K}{x+1} \leq \frac{K}{\sqrt{x^2+1}} \leq \frac{\CG}{\sqrt{x^2+1}} \Big(2G(x)+\frac{1}{x+1}\Big).
	\end{equation*}
	This justifies \eqref{eq:8} and finishes the proof.
\end{proof}

We are now ready to prove the lower bound for $\hOC_t(x,y)$.

\begin{prop}\label{prop:lower_bound}
	The Opdam--Cherednik heat kernel satisfies the lower bound
	\begin{equation*}
		\hOC_t (x,y) \gtrsim H_t^0(x,y), 
	\end{equation*}
	uniformly in $t\geq 1$ and $x,y>0$.
\end{prop}
\begin{proof}
	Since the kernels $\hOC_t$ and $H_t^0$ are symmetric, it suffices to prove the bound for $0<y\le x<\infty$. Moreover, by Lemma \ref{lm:lowercompar} we are reduced to showing that
	\begin{equation*}
		\hOC_t(x,y) \gtrsim H_t^-(x,y),\qquad 0<y\le x<\infty,\ t\geq 1.
	\end{equation*}
	
	Observe that
	\begin{align}\label{eq:2}
		\begin{split}
		\eth_x H_t^{\psi_1} (x,y) &= -\frac{x-y}{2t} + \frac{\varphi_0'(x)}{\varphi_0(x)} +\frac{k t}{x(t+xy)} - \frac{1}{x},\\
		\eth_x H_t^{\psi_2^-}(x,y) &=-\frac{x-y}{2t} +\frac{\varphi_0'(x)}{\varphi_0(x)} - \frac{y}{2t+xy} + \frac{\CG x}{t\sqrt{x^2+1}},
		\end{split} 
	\end{align}
	and
	\begin{equation}\label{eq:3}
		D \chi = -\frac{x}{t^2} \chi' -\frac{\chi''}{t^2} -\frac{k_1\coth(\frac{x}{2}) \chi'}{t}-\frac{2k_2\coth(x) \chi'}{t}.
	\end{equation}
	Note that $\chi$ and its derivatives are always evaluated at $x/t$.
	
	Thus, \eqref{eq:D(product)} implies
	\begin{equation*}
		\begin{split}
		&D H^-_t= C_1^-(1-\chi) D H_t^{\psi_1} + \chi D H_t^{\psi_2^-} +\frac{\chi''}{t^2}(C_1^- H_t^{\psi_1}-H_t^{\psi_2^-})\\
		&\ + C_1^-\chi' H_t^{\psi_1}\Big( \frac{x}{t^2}+\frac{k_1\coth \frac{x}{2}}{t} +\frac{2k_2\coth x}{t} -\frac{x-y}{t^2} + \frac{2\varphi'_0(x)}{t\varphi_0(x)}+\frac{2k}{x(t+xy)}-\frac{2}{tx}\Big)\\
		&\ -\chi' H_t^{\psi_2^-} \Big( \frac{x}{t^2}+\frac{k_1\coth \frac{x}{2}}{t} + \frac{2k_2\coth x}{t} -\frac{x-y}{t^2} +\frac{2\varphi'_0(x)}{t\varphi_0(x)} - \frac{2y}{t(2t+xy)} \\
		&\qquad+\frac{2 \CG x}{t^2 \sqrt{x^2+1}}\Big).
		\end{split}
	\end{equation*}
	Hence,
	\begin{equation*}
		\begin{split}
			\frac{D H_t^-}{H_t^-} &=\frac{C_1^-(1-\chi)H_t^{\psi_1}}{H_t^-} \cdot \frac{D H_t^{\psi_1}}{H_t^{\psi_1}} + \frac{\chi H_t^{\psi_2^-}}{H_t^-}\cdot \frac{D H_t^{\psi_2^-}}{H_t^{\psi_2^-}} + \frac{\chi''}{t^2}\cdot \frac{C_1^-H_t^{\psi_1}-H_t^{\psi_2^-}}{H_t^-}\\
			&\qquad+ \frac{y}{t^2} \cdot\frac{\chi' (C_1^-H_t^{\psi_1}-H_t^{\psi_2^-})}{H_t^-} + \frac{C_1^-\chi' H_t^{\psi_1}}{H_t^-}\Big( \frac{2(G(x)-1)}{tx} + \frac{2k}{x(t+xy)}\Big)\\
			&\qquad -\frac{\chi' H_t^{\psi_2^-}}{H_t^-}\Big( \frac{2(G(x)-1)}{tx} +\frac{2}{tx} -\frac{2y}{t(2t+xy)} +\frac{2\CG x}{t^2\sqrt{x^2+1}}\Big).
		\end{split}
	\end{equation*}
	
	Notice that $\chi',\chi''\neq 0$ if, and only, if $t\le x\le 2t$. In that range we have
	\begin{equation*}
		H_t^{\psi_1}(x,y)\simeq H_t^{\psi_2^-}(x,y)\simeq H_t^-(x,y).
	\end{equation*}
	 Moreover, $|\chi'|,|\chi''|\lesssim 1$ with the underlying constants independent of $t,x,y$. Thus,
	\begin{equation}\label{eq:4}
		\frac{|\chi'| H_t^{\psi_1}}{H_t^-}\cdot\Big| \frac{2(G(x)-1)}{tx} + \frac{2k}{x(t+xy)}\Big|\lesssim \frac{|\chi'|}{tx}\lesssim t^{-2},
	\end{equation}
	and
	\begin{equation*}
		\frac{|\chi'| H_t^{\psi_2^-}}{H_t^-} \cdot \Big| \frac{2(G(x)-1)}{tx} +\frac{2}{tx} -\frac{2y}{t(2t+xy)} +\frac{2\CG x}{t^2\sqrt{x^2+1}}\Big|\lesssim \frac{|\chi'|}{tx} + \frac{|\chi'|}{t^2}\lesssim t^{-2}.
	\end{equation*}
	Moreover,
	\begin{equation*}
		\frac{|\chi''|}{t^2}\cdot \frac{|C_1^-H_t^{\psi_1}-H_t^{\psi_2^-}|}{H_t^-}\lesssim t^{-2}.
	\end{equation*}
	
	A key observation is that $\chi'\leq 0$ and for $1\le t\le x\le 2t$ it holds that
	\begin{equation*}
		\frac{H_t^{\psi_2^-}}{C_1^- H_t^{\psi_1}} =\frac{(t+xy)^k e^{\frac{\CG \sqrt{x^2+1}}{t}}}{C_1^- (x+xy)^{k-1} (2t+xy)}\leq \frac{e^{\sqrt{5}\CG}}{C_1^-} \cdot\Big(\frac{t+xy}{x+xy}\Big)^k\leq 1, 
	\end{equation*}
	where we used the definition \eqref{eq:C_1-} of $C^-_1$. Thus, $\chi'(C_1^-H_t^{\psi_1}-H_t^{\psi_2^-})$ is negative.
	
	Combining the above, together with Lemmas \ref{lm:Hpsi1} and \ref{lm:Hpsi2-} we obtain
	\begin{equation*}
		\frac{D H_t^-(x,y)}{H_t^-(x,y)} = J_{1}(x,y,t) + J_2(x,y,t),
	\end{equation*}
	where $|J_1| \lesssim t^{-3/2}$ and $J_2\leq 0$.

	Now let us denote
	\begin{equation*}
		H_t(x,y) = A_1\hOC_t(x,y) - e^{\frac{A_2}{\sqrt{t}}} H_t^-(x,y),
	\end{equation*}
	where $A_1,A_2$ are positive constants to be determined later. Our task is to justify that $H_t(x,y)\geq 0$ for $t\geq 1$ and $0<y\le x<\infty$. For that purpose we apply the parabolic minimum principle, i.e. Lemma \ref{lm:PMP}, to $H_t$, on the set $[y,\infty)\times [1,T]$, where $y>0$ is fixed and $T\geq 1$ is arbitrarily large.
	
	To verify the required assumptions observe that
	\begin{equation*}
		\frac{D H_t(x,y)}{e^{\frac{A_2}{\sqrt{t}}} H_t^-(x,y) }= \frac{A_2}{2 t^{3/2}} - \frac{D H_t^-(x,y)}{H_t^-(x,y)}\geq \frac{A_2}{2 t^{3/2}} - J_1(x,y,t).
	\end{equation*}	
	Since $J_1=O(t^{-3/2})$, for $A_2$ large enough we have $D H_t^-(x,y)\geq 0$ in $(x,t)\in [y,\infty)\times [1,\infty)$.
	To check the boundary conditions observe that, by Propositions \ref{prop:small_t} and \ref{prop:diagonal}, and Lemma \ref{lm:lowercompar}, the kernel $\hOC_t(x,y)$ is comparable with $H_t^-(x,y)$ for $t=1$ or $x=y$, thus
	\begin{equation*}
		H_1(x,y) \geq 0 \qquad \text{and} \qquad H_t(x,x)\geq 0
	\end{equation*}
	for $0<y\le x<\infty,\ t\geq 1$, and $A_1$ large enough. Therefore, $H_t(x,y)\geq 0$ in that range, and the proof is completed.
	
\end{proof}

\subsection{Large time, upper bound}

Let $\chi$ be as previously. We define
\begin{equation*}
	H_t^+(x,y) = C_1^+ \Big[1-\chi\Big(\frac{x}{t}\Big)\Big] H_t^{\psi_1}(x,y) + \chi\Big(\frac{x}{t}\Big)  H_t^{\psi_2^+}(x,y) ,
\end{equation*}
where
\begin{equation}\label{eq:C_1+}
	C_1^+=\frac{1-e^{-1}}{\max(2^{k-1},1)},
\end{equation}
 $\psi_1 $ is defined in \eqref{eq:psi1_psi2-}, and $\psi_2^+=\psi_{2,1}^+ + \psi_{2,2}^+$, where
\begin{align*}
	\psi_{2,1}^+(x,y,t) &= \frac{1}{t} e^{\frac{-2xy+y^2}{4t}} e^{-\frac{2\CG\sqrt{x^2+1}}{t} } e^{\frac{4\sqrt{5}\CG+1}{2}},\\
	\psi_{2,2}^+(x,y,t) &= \frac{1-e^{-\frac{xy}{t}}}{xy}.
\end{align*}

We begin by showing that $H_t^+$ is comparable with the desired bound. Much as above, it is sufficient to consider $y\le x$.

\begin{lm}\label{lm:uppercompar}
	It holds that
	\begin{equation*}
		H_t^+(x,y)\simeq H_t^0(x,y),
	\end{equation*}
	uniformly in $t\geq 1$ and $0<y\le x<\infty$.
\end{lm}
\begin{proof}
	It suffices to prove that
	\begin{equation*}
	\psi_2^+ (x,y,t)\simeq \frac{1}{t+xy},
	\end{equation*}
	uniformly in $t\geq 1$ and $0<y\le x \le 2t$, and then proceed much as in the proof of Lemma \ref{lm:lowercompar}.
	
	We consider two cases. Firstly, assume that $xy\geq t$. By the simple inequality $e^{-s}\leq s^{-1}$, we obtain 
	\begin{equation*}
		0\leq \psi_{2,1}^+(x,y,t)  \leq \frac{4 e^{\frac{4\sqrt{5}\CG+1}{2}}}{xy}\qquad \text{and}\qquad \frac{1-e^{-1}}{xy}\le \psi_{2,2}^+(x,y,t) \le \frac{1}{xy}.
	\end{equation*}
	Secondly, assume $xy\le t$. This time by the inequality $1-e^{-s}\leq s$, $s>0$, we get
	\begin{equation*}
		\frac{1}{t}\le\psi_{2,1}^+(x,y,t)\le e^{ \frac{4\sqrt{5}\CG+1}{2} }\cdot \frac{1}{t} \qquad\text{and} \qquad 0\le \psi_{2,2}^+(x,y,t) \leq \frac{1}{t}.
	\end{equation*}
	
	Combining the above we conclude
	\begin{equation}\label{eq:5}
		\frac{1-e^{-1}}{t+xy}\le\psi_2^+(x,y,t)\leq \frac{2(4e^{ \frac{4\sqrt{5}\CG+1}{2} }+1)}{t+xy}.
	\end{equation}
\end{proof}

Much as for the lower bound, we study the action of the heat operator $D$ on components of $H_t^+$.

\begin{lm}\label{lm:Hpsi2+}
	Let $t\geq 1$ and $0<y\le x<\infty$ be such that $x\le 2t$. The following identity holds:
	\begin{equation}\label{eq:9}
		\frac{D H_t^{\psi_2^+}(x,y)}{H_t^{\psi_2^+}(x,y)}=I_1(x,y,t) +I_2(x,y,t),
	\end{equation}
	where $|I_1(x,y,t)|\lesssim t^{-3/2}$ and $I_2(x,y,t)\geq 0$.
\end{lm}
\begin{proof}
	We start by computing $D H_t^{\psi_{2,i}^+}/H_t^{\psi_{2,i}^+}$, $i=1,2$. For that purpose observe that
	\begin{align*}
		\eth_t \psi_{2,1}^+ (x,y,t) &= -\frac{1}{t} +\frac{2xy-y^2}{4t^2} +\frac{2\CG \sqrt{x^2+1}}{t^2},\\
		\eth_x \psi_{2,1}^+(x,y,t) &= -\frac{y}{2t} -\frac{2\CG x}{t\sqrt{x^2+1}},\\
		\eth_x^2 \psi_{2,1}^+ (x,y,t) &= \frac{y^2}{4t^2} +\frac{2\CG xy}{t^2 \sqrt{x^2+1}}+ \frac{4\CG^2 x^2}{t^2 (x^2+1)}-\frac{2\CG}{t(x^2+1)^{3/2}}.
	\end{align*}
	Thus, by Lemma \ref{lm:gen_form} we have
	\begin{equation*}
		\begin{split}
		\frac{D H_t^{\psi_{2,1}}}{H_t^{\psi_{2,1}}}&= \big( G(x)-1\big) \Big( \frac{x-y}{xt} +\frac{y}{xt} +\frac{4\CG}{t\sqrt{x^2+1}}\Big)  -\frac{1}{t} +\frac{2xy-y^2}{4t^2} +\frac{2\CG \sqrt{x^2+1}}{t^2}\\
		 &\qquad+\frac{x-y}{xt} -\Big(\frac{x-y}{t} -\frac{2}{x}\Big) \Big(\frac{y}{2t} +\frac{2\CG x}{t\sqrt{x^2+1}}\Big) -\frac{y^2}{4t^2} -\frac{2\CG xy}{t^2 \sqrt{x^2+1}}\\
		 &\qquad- \frac{4\CG^2 x^2}{t^2 (x^2+1)}+\frac{2\CG}{t(x^2+1)^{3/2}}\\
		 &=\frac{1}{t}\big(G(x)-1\big)+\frac{2\CG}{t\sqrt{x^2+1}}\Big( 2G(x)+\frac{1}{x^2+1}\Big) +\frac{2\CG}{t^2\sqrt{x^2+1}} - \frac{4\CG^2 x^2}{t^2(x^2+1)}.
		\end{split}
	\end{equation*}
	
	Similarly, for $\psi_{2,2}^+$ we have
	\begin{align*}
		\eth_t \psi_{2,2}^+ (x,y,t) &= -\frac{xy e^{-\frac{xy}{t}}}{t^2(1-e^{-\frac{xy}{t} } )},\\
		\eth_x \psi^+_{2,2}(x,y,t) &= -\frac{1}{x} +\frac{y e^{-\frac{xy}{t} }}{t(1-e^{-\frac{xy}{t} })}=-\frac{1}{x} -\frac{y}{t}+\frac{y }{t(1-e^{-\frac{xy}{t} })},\\
		\eth_x^2 \psi_{2,2}^+(x,y,t)&= \frac{2}{x^2}+\frac{y^2}{t^2}+\frac{2y}{xt}-\frac{y^2}{t^2 (1-e^{-\frac{xy}{t} })}-\frac{2y}{tx (1-e^{-\frac{xy}{t} })}.
	\end{align*}
	Hence, Lemma \ref{lm:gen_form} implies
	\begin{equation*}
		\begin{split}
		\frac{D H_t^{\psi_{2,2}^+}}{H_t^{\psi_{2,2}^+}}&= \big(G(x)-1)\Big(\frac{x-y}{xt} +\frac{2}{x^2} -\frac{2y e^{ -\frac{xy}{t}}}{tx (1-e^{-\frac{xy}{t} })}\Big) -\frac{xy e^{-\frac{xy}{t}}}{t^2(1-e^{-\frac{xy}{t} } )}\\
		&\qquad +\frac{x-y}{xt} -\Big(\frac{x-y}{t}-\frac{2}{x}\Big)\Big(\frac{1}{x}+\frac{y}{t}-\frac{y}{t(1-e^{-\frac{xy}{t} } ) }\Big) -\frac{2}{x^2}-\frac{y^2}{t^2}-\frac{2y}{xt}\\
		&\qquad+\frac{y^2}{t^2 (1-e^{-\frac{xy}{t} })}+\frac{2y}{tx (1-e^{-\frac{xy}{t} })}\\
		&=\big(G(x)-1\big)\frac{1}{t} \Big(1+\frac{2t}{x^2} -\frac{y\coth(\frac{xy}{2t})}{x} \Big).
		\end{split}
	\end{equation*}
	
	Let
	\begin{equation*}
		M(x,y,t):=1+\frac{2t}{x^2} -\frac{y\coth(\frac{xy}{2t})}{x} =\frac{x-y}{x} +\frac{y}{x}\Big[1+\frac{2t}{xy}-\coth\Big(\frac{xy}{2t}\Big)\Big].
	\end{equation*}
	Notice that, by the inequality $0<1+1/s-\coth(s)<1$, $s>0$, we obtain
	\begin{equation}\label{eq:M}
		0\leq M(x,y,t)\leq 1.
	\end{equation}
	
	By combining the above we obtain
	\begin{multline*}
			\frac{D H_t^{\psi_2^+}}{H_t^{\psi_2^+}} = \frac{\psi_{2,1}^+}{\psi_2^+} \Big[ \big(G(x)-1\big)\frac{1}{t} + \frac{2\CG}{t\sqrt{x^2+1}}\Big(2G(x)+\frac{1}{x^2+1}\Big)\Big]\\
		+\frac{\psi_{2,2}^+}{\psi_2^+} \big(G(x)-1\big) \frac{M(x,y,t)}{t}+\frac{\psi_{2,1}^+}{\psi_2^+}\Big( \frac{2\CG}{t^2\sqrt{x^2+1}} -\frac{4\CG^2 x^2}{t^2(x^2+1)}\Big).
	\end{multline*}
	Denote
	\begin{equation*}
		\begin{split}
			I(x,y,t)=\frac{\psi_{2,1}^+}{t\psi_2^+} \Big[ \big(G(x)-1\big)\Big(1+\frac{\psi_{2,2}^+}{\psi_{2,1}^+} M(x,y,t) \Big)+ \frac{2\CG}{\sqrt{x^2+1}}\Big(2G(x)+\frac{1}{x^2+1}\Big)\Big].
		\end{split}
	\end{equation*}
	Clearly, the remaining summands in the formula for $D H_t^{\psi_2^+}/H_t^{\psi_2^+}$ are $O(t^{-2})$. 
	
	We consider two cases. Firstly, assume $xy\geq t$. Then, $x\geq \sqrt{t}$ and, by Lemma \ref{lm:G} and \eqref{eq:M} we obtain
	\begin{equation*}
	\big|I(x,y,t)\big|\lesssim \frac{1}{t(x+1)}=O(t^{-3/2}).
	\end{equation*}
	Therefore, \eqref{eq:9} holds.
	
	Secondly, if $xy\le t$, then (recall that $x\le 2t$)
	\begin{equation*}
		\frac{\psi_{2,2}^+}{\psi_{2,1}^+} = \frac{(1-e^{-\frac{xy}{t}}) t e^{\frac{2xy-y^2}{4t}} e^{\frac{2\CG \sqrt{x^2+1}}{t} }}{xy e^{\frac{4\sqrt{5}\CG+1}{2} }} \leq \frac{e^{\frac{xy}{2t}}}{\sqrt{e}}\leq 1.
	\end{equation*}
	Thus, Lemma \ref{lm:G}, \eqref{eq:M}, and the definition \eqref{eq:CG} of $\CG$ imply
	\begin{equation*}
		\Big| \big(G(x)-1\big)\Big(1+\frac{\psi_{2,2}^+}{\psi_{2,1}^+} M(x,y,t) \Big)\Big|\leq \frac{2K}{x+1} \leq \frac{2K}{\sqrt{x^2+1}}\leq \frac{2\CG}{\sqrt{x^2+1}}\Big(2G(x)+\frac{1}{x^2+1}\Big).
	\end{equation*}
	This shows that $I(x,y,t)\geq 0$ and \eqref{eq:9} holds as well.
		
\end{proof}

We will make use of the following bound on $\eth_x \psi_2^+$.

\begin{lm}\label{lm:psi_21+psi_22}
	We have
	\begin{equation*}
	\big| \eth_x \psi_2^+(x,y,t)\big| \lesssim t^{-1},
	\end{equation*}
	uniformly in $t\geq 1$ and $0<y\le x<\infty$ satisfying $t\le x\le 2t$.
\end{lm}
\begin{proof}
	Directly from the definition of $\psi_{2,1}^+$ and $\psi_{2,2}^+$ we have
	\begin{equation*}
		\partial_x \psi_{2,1}^+ + \partial_x\psi_{2,2}^+= -\Big(\frac{y}{2t}+\frac{2\CG x}{t\sqrt{x^2+1}}\Big)\psi_{2,1}^+ + \frac{e^{-\frac{xy}{t}}}{xt}-\frac{1-e^{-\frac{xy}{t}}}{x^2y}.
	\end{equation*}
	
	As usually, we consider two cases. Firstly, assume $xy\le t$ and, consequently, $y\le 1$. Then, by applying $\psi_{2,1}^+\lesssim t^{-1}$ we obtain
	\begin{equation*}
	\big| \partial_x \phi_{2,1}^+ + \partial_x\psi_{2,2}^+\big|\lesssim  \frac{1}{t^2}+\frac{1}{tx}\lesssim \frac{1}{t(t+xy)}. 
	\end{equation*}
	Secondly, if $xy\geq t$, then we use $\psi_{2,1}^+\lesssim t/(xy)^2$ and obtain
	\begin{equation*}
	\big| \partial_x \phi_{2,1}^+ + \partial_x\psi_{2,2}^+\big|\lesssim  \frac{1}{x^2y}+\frac{1}{x^2 y^2}\lesssim \frac{1}{t(t+xy)}. 
	\end{equation*}
	
	In the proof of Lemma \ref{lm:uppercompar} we showed that in the considered range of $x,y,t$ we have $\psi_2^+\simeq (t+xy)^{-1}$. Thus,
	\begin{equation*}
		\big| \eth_x\psi_2^+\big| \lesssim \frac{t+xy}{t(t+xy)}=\frac{1}{t},
	\end{equation*}
	which finishes the proof.
\end{proof}

We are now ready to prove the upper bound for $\hOC_t(x,y)$.

\begin{prop}\label{prop:upper_bound}
	The Opdam--Cherednik heat kernel satisfies the upper bound
	\begin{equation*}
		\hOC_t (x,y) \lesssim H_t^0(x,y), 
	\end{equation*}
	uniformly in $t\geq 1$ and $x,y>0$.
\end{prop}
\begin{proof}
	By the symmetry of $\hOC_t$ and $H_t^0$,  Lemma \ref{lm:uppercompar}, it suffices to justify that 
	\begin{equation*}
		\hOC_t (x,y)\lesssim H_t^+(x,y),\qquad 0<y\le x<\infty,\ t\geq 1.
	\end{equation*}
	
	Much as in the proof of Proposition \ref{prop:lower_bound}, by \eqref{eq:2}, \eqref{eq:3}, and
	\begin{equation*}
		\eth_x H_t^{\psi_2^+}(x,y)=-\frac{x-y}{2t} +\frac{\varphi'_0(x)}{\varphi_0(x)} +\eth_x\psi_2^+ 
	\end{equation*}
	we have
	\begin{equation*}
		\begin{split}
			\frac{D H_t^+}{H_t^+}&= \frac{C_1^+(1-\chi) D H_t^{\psi_1}}{H_t^+} + \frac{\chi D H_t^{\psi_2^+}}{H_t^+} +\frac{\chi''}{t^2}\frac{ C_1^+ H_t^{\psi_1}-H_t^{\psi_2^+}}{H_t^+}\\
			&\hspace{-1cm} + \frac{C_1^+ \chi' H_t^{\psi_1}}{H_t^+}\Big( \frac{x}{t^2} +\frac{k_1\coth \frac{x}{2}}{t}+\frac{2k_2\coth x}{t} -\frac{x-y}{t^2} + \frac{2\varphi'_0(x)}{t\varphi_0(x)}+\frac{2k}{x(t+xy)}-\frac{2}{tx}\Big)\\
			&\hspace{-1cm} -\frac{\chi' H_t^{\psi_2^+}}{H_t^+} \Big( \frac{x}{t^2} +\frac{k_1\coth \frac{x}{2}}{t}+ \frac{2k_2\coth x}{t} -\frac{x-y}{t^2} +\frac{2\varphi'_0(x)}{t\varphi_0(x)} + \frac{2\eth_x \psi_2^+}{t} \Big)\\
			&=\frac{C_1^+ (1-\chi) D H_t^{\psi_1}}{H_t^+} + \frac{\chi D H_t^{\psi_2^+}}{H_t^+} +\frac{\chi''}{t^2}\frac{ C_1^+ H_t^{\psi_1}-H_t^{\psi_2^+}}{H_t^+}\\
			&\hspace{-1cm}+ \frac{y\chi'}{t^2} \cdot\frac{ C_1^+ H_t^{\psi_1}-H_t^{\psi_2^+}}{H_t^+} + \frac{C_1^+\chi' H_t^{\psi_1}}{H_t^+}\Big( \frac{2(G(x)-1)}{tx} + \frac{2k}{x(t+xy)}\Big)\\
			&\hspace{-1cm} -\frac{\chi' H_t^{\psi_2^+}}{H_t^+}\Big( \frac{2(G(x)-1)}{tx} +\frac{2}{tx} +\frac{2\eth_x\psi_2^+}{t}\Big).
		\end{split}
	\end{equation*}
	
	Recall that $\chi',\chi''\neq0$ if, and only, if $t\le x\le 2t$. In this range we have $H_t^{\psi_1}\simeq H_t^{\psi_2^+}\simeq H_t^+$. Much as in \eqref{eq:4} we have
	\begin{equation*}
		 \frac{|\chi'| H_t^{\psi_1}}{H_t^+}\Big| \frac{2(G(x)-1)}{tx} + \frac{2k}{x(t+xy)}\Big|\lesssim \frac{|\chi'|}{tx}\lesssim t^{-2}.
	\end{equation*}
	Similarly, by using additionally Lemma \ref{lm:psi_21+psi_22} we obtain
	\begin{equation*}
		 \frac{|\chi'| H_t^{\psi_2^+}}{H_t^+}\Big| \frac{2(G(x)-1)}{tx} +\frac{2}{tx} +\frac{2\eth_x\psi_2^+}{t}\Big|\lesssim t^{-2}.
	\end{equation*}
	Furthermore,
	\begin{equation*}
		\Big| \frac{\chi''}{t^2}\frac{ C_1^+ H_t^{\psi_1}-H_t^{\psi_2^+}}{H_t^+}\Big|\lesssim t^{-2}. 
	\end{equation*}
	Lastly, we claim that
	\begin{equation*}
		\chi'\Big(\frac{x}{t}\Big) \big(C_1^+ H_t^{\psi_1}(x,y)-H_t^{\psi_2^+}(x,y)\big)\geq 0.
	\end{equation*}
	Indeed, this follows from the fact that $\chi'\le 0$ and that for $t\le x\le 2t$ it holds that
	\begin{equation*}
		\frac{C_1^+ \psi_1(x,y,t)}{\psi_2^+(x,y,t)} \leq \frac{C_1^+}{1-e^{-1}} \Big(\frac{x+xy}{t+xy}\Big)^{k-1} \leq  1,
	\end{equation*}
	where we used \eqref{eq:5} and the definition \eqref{eq:C_1+} of $C_1^+$.

	Combining the above together with Lemmas \ref{lm:Hpsi1} and \ref{lm:Hpsi2+} we arrive at
	\begin{equation*}
		\frac{D H_t^+(x,y)}{H_t^+(x,y)}=J_1(x,y,t)+J_2(x,y,t),
	\end{equation*}
	where $J_1=O(t^{-3/2})$ and $J_2\ge 0$.
	
	Now denote
	\begin{equation*}
		H_t(x,y):=-\frac{1}{A_1}\hOC_t(x,y) + e^{-\frac{A_2}{\sqrt{t}}}H_t^+(x,y),
	\end{equation*}
	where $A_1,A_2$ are (large) positive constants. Much as in Proposition \ref{prop:lower_bound}, we apply Lemma \ref{lm:PMP} to $H_t$ on $[y,\infty)\times[1,T)$ for a fixed $y>0$ and arbitrarily large $T$.
	
	The necessary assumptions are satisfied:
	\begin{equation*}
		\frac{D H_t(x,y)}{e^{-\frac{A_2}{\sqrt{t}}} H_t^+(x,y)} = \frac{A_2}{2t^{3/2}} + \frac{D H_t^+(x,y)}{H_t^+(x,y)}\geq \frac{A_2}{2t^{3/2}} + J_1(x,y,t)\geq 0
	\end{equation*}
	 for $A_2$ large enough. Moreover, by Propositions \ref{prop:small_t} and \ref{prop:diagonal}, and Lemma \ref{lm:uppercompar}, the kernels $\hOC_t(x,y)$ and $H^+_t(x,y)$ are comparable for $t=1$ or $x=y$. Thus,
	 \begin{equation*}
	 	H_1(x,y)\geq 0\qquad \text{and} \qquad H_t(x,x)\geq 0
	 \end{equation*}
	for $0<y\le x<\infty$, $\geq 1$, and $A_1$ large enough. Thus, $H_t(x,y)\geq 0$ in that range, which finishes the proof.
	
\end{proof}

\subsection{Conclusion}

\begin{proof}[Proof of Theorem \ref{mainthm}]
	The bound on $\hOC_t(x,y)$ for $t,x,y>0$ follows immediately from Propositions \ref{prop:small_t}, \ref{prop:lower_bound}, and \ref{prop:upper_bound}.
\end{proof}

\begin{proof}[Proof of Theorem \ref{thm:Jacobi}]
	The claim follows immediately from Theorem \ref{mainthm} and \eqref{eq:OC-Jacobi}.
\end{proof}

The obtained bound for $\hOC_t(x,y)$ in rank $1$ suggests, that the conjecture in general rank posted in \cite{GraSaw25} should be reformulated as
\begin{multline*}
	\hOC_t(x,y)\simeq t^{-n/2} e^{-\frac{|x-y|^2}{4t}} e^{-|\rho|^2 t} e^{-\langle \rho, x+y\rangle}\\
	 \times\prod_{\alpha\in\Sigma^{++}} \frac{(1+\langle  \alpha,x\rangle) (1+\langle  \alpha,y\rangle) \big[t+(1+\langle  \alpha,x\rangle)(1+\langle \alpha,y\rangle)\big]^{k_\alpha + k_{2\alpha}-1}}{\big(t+\langle \alpha,x\rangle \langle \alpha,y\rangle \big)^{k_\alpha + k_{2\alpha}}}.
\end{multline*}

\end{document}